\documentclass[12pt]{article}
\usepackage{setspace}
\usepackage[dvips]{color}
\usepackage{geometry}
\usepackage{graphicx}
\usepackage{amsmath,amsthm}
\usepackage{amssymb}
\usepackage{multicol}
\usepackage{ulem}
\usepackage{amsfonts}
\usepackage{mathrsfs}
\usepackage[pdfstartview=FitH]{hyperref}
 \newtheorem{thm}{Theorem}[section]
 \newtheorem{cor}[thm]{Corollary}
 \newtheorem{lem}[thm]{Lemma}
 \newtheorem{prop}[thm]{Proposition}
 
 \theoremstyle{definition}
 \newtheorem{defn}[thm]{Definition}

 \newtheorem{nott}{Notation}
 \newtheorem{rem}{Remark}
  \DeclareMathAlphabet{\mathsfsl}{OT1}{cmss}{m}{sl}

    \newcommand{\FH}{\mathfrak{H}}

 \newcommand{\Cnum}{\mathbb{C}}
 
 \newcommand{\Nnum}{\mathbb{N}}
 \newcommand{\mi}{\mathrm{i}}
 \newcommand{\dif}{\mathrm{d}}

 \newcommand{\abs}[1]{\left\vert#1\right\vert}
 \newcommand{\set}[1]{\left\{#1\right\}}
 \newcommand{\norm}[1]{\left\Vert#1\right\Vert}
 \newcommand{\innp}[1]{\langle {#1}\rangle}

\allowdisplaybreaks

\title{Product Formula, Independence and Asymptotic Moment-Independence for Complex Multiple Wiener-It\^{o} Integrals }
\author{\rm\small
\noindent CHEN Yong\\
\noindent \footnotesize School of Mathematics and Computing Science, Hunan
University of Science and Technology,\\
\noindent \footnotesize Xiangtan, Hunan, 411201,
P.R.China. zhishi@pku.org.cn\\
}
\date{}
\begin{document}
\maketitle
\maketitle \noindent {\bf Abstract } We present a product formula for complex multiple Wiener-It\^{o} integrals. As applications, we show \"{U}st\"{u}nel-Zakai independence criterion and the Nourdin-Rosi\'{n}ski asymptotic moment-independent criterion for complex multiple Wiener-It\^{o} integrals.
\maketitle
\section{Introduction}
The product formula of real multiple Wiener-It\^{o} integrals is well known. Using this formula there are many interesting findings such as \"{U}st\"{u}nel-Zakai independence criterion \cite{klb,uzk} for two multiple Wiener-It\^{o} integrals, Nourdin-Rosi\'{n}ski asymptotic moment-independence criterion between blocks consisting of multiple Wiener-It\^{o} integrals \cite{nr} and Fourth Moment Theorem (or say: Nualart-Peccati criterion) of a normalized sequence of real multiple Wiener-It\^{o} integrals in a fixed Wiener chaos \cite{np,NuaPec}.

Both real multiple Wiener-It\^{o} integrals and complex multiple Wiener-It\^{o} integrals (see \cite{ito} or Section~\ref{plr} below) were established by It\^{o} K. almost at the same time \cite{ito2,ito}. However, the product formula of complex multiple Wiener-It\^{o} integrals is still unknown up to now as far as we know. The key aim of this paper is to answer this question (see Theorem~\ref{pdt fml}). As far as we know, there exist at least three different approaches to prove the product formula. In this paper, we adopt the most simple one by using the relationship between complex multiple Wiener-It\^{o} integrals and complex Hermite polynomials given by It\^{o} \cite{ito}.

As applications, we will show \"{U}st\"{u}nel-Zakai independence criterion, i.e., a necessary and sufficient condition on the pair of kernels $(f,g)$ is derived under which the complex multiple Wiener-It\^{o} integrals $I_{a,b}(f),\,I_{c,d}(g)$ are independent (see Theorem~\ref{thm2}). In the finial, by using the connection between real Wiener-It\^{o} integrals and complex Wiener-It\^{o} integrals \cite[Theorem 3.3]{cl2}, we list two related results as an appendix: Nourdin-Rosi\'{n}ski asymptotic moment-independence criterion and joint convergence criterion for d-dimensional vectors consisting of complex multiple Wiener-It\^{o} integrals (see Theorem~\ref{thm3}, Corollary~\ref{cor4}).
\section{Preliminaries}\label{plr}
In this section,  we shortly recall the theory of complex multiple Wiener-It\^{o} integrals of It\^{o} \cite{ito}.
 Consider a triple $(T,\mathcal{B},\mu)$, where the measure $\mu$ is positive, $\sigma$-finite and non-atomic. $\FH=L^2(T,\mathcal{B},\mu)$ is a complex separable Hilbert space. A complex Gaussian random measure over $(T,\mathcal{B})$, with control $\mu$, is a centered complex Gaussian family of the type
   \begin{equation*}
     \mathbf{M}=\set{M(B):B\in \mathcal{B},\mu(B)<\infty},
   \end{equation*}
 such that, for every $B, C\in \mathcal{B}$ with finite measure,
   \begin{equation*}
      E[M(B)\overline{M(C)}]=\mu(B\cap C).
   \end{equation*}

   \begin{nott}
For a fixed $(p,q)$, suppose that $f\in \FH^{\otimes (p+ q)}$. $\hat{f}$ is the symmetrization of $f$ in the sense of It\^{o} \cite{ito}:
\begin{equation}\label{ito-sense}
   \tilde{f}(t_1,\dots,t_{p+q})=\frac{1}{p!q!}\sum_{\pi}\sum_{\sigma}f(t_{\pi(1)},\dots,t_{\pi(p)},t_{\sigma(1)},\dots,t_{\sigma(q)}),
\end{equation}
where $\pi$ and $\sigma$ run over all permutations of $(1,\dots,p)$ and $(p+1,\dots, p+q)$ respectively.
Denote by $\FH^{\odot p}\otimes \FH^{\odot q}=L^2_S(T^p,\mathcal{B}^{\otimes p},\mu^{\otimes p})\otimes L^2_S(T^q,\mathcal{B}^{\otimes q},\mu^{\otimes q})$ the space of square integrable and symmetric functions on $T^{p+q}$ in the above sense. Notice that (\ref{ito-sense}) is different to the ordinary symmetrization of $f$ in the theory of real multiple integrals which is given by
\begin{equation}\label{ordinary-sense}
   \hat{f}(t_1,\dots,t_{p+q})=\frac{1}{(p+q)!}\sum_{\pi}f(t_{\pi(1)},\dots,t_{\pi(p+q)}),
\end{equation}
where $\pi$ runs over all permutations of $(1,\dots,p+q)$.
\end{nott}
Obviously, we have that (see (5.2) of \cite{ito})
\begin{equation}\label{daxiao}
  \norm{ \tilde{f}}\le \norm{f}.
\end{equation}
 \begin{defn}{\bf(Complex multiple Wiener-It\^{o} integrals \cite{ito})}\label{dff}
 Suppose that $E_1,\dots,E_n\subset \mathcal{B} $ is any system of disjoint sets and $e_{i_1\dots i_p j_1\dots j_q}$ is a complex-valued function defined for $i_1,\dots, i_p, j_1,\dots, j_q=1,2,\dots,n$ such that $e_{i_1\dots i_p j_1\dots j_q}=0$ unless $i_1,\dots, i_p, j_1,\dots, j_q$ are all different.
 Let $\mathcal{S}_{pq}$ denote the set of all functions of the form
\begin{equation}\label{sfc}
   f(t_1,\dots,t_p,s_1,\dots,s_q)=\sum e_{i_1\dots i_p j_1\dots j_q} \mathbf{1}_{E_{i_1}\times\dots\times E_{i_p}\times E_{j_1}\times\dots\times E_{j_q} },
\end{equation}
where $\mathbf{1}_{B}(\cdot) $ is the characteristic function of the set $B$. The multiple integral of $f $ is defined by
\begin{equation}
   I_{p,q}(f)=\sum e_{i_1\dots i_p j_1\dots j_q} M(E_{i_1})\dots M(E_{i_p})\overline{M(E_{j_1})}\dots \overline{ M(E_{j_q})}.
\end{equation}
Clearly, the above integral satisfies that
\begin{align}
   I_{p,q}(f)&=I_{p,q}(\tilde{f}),\label{ipq}\\
   E[I_{p,q}(f) \overline{I_{p,q}(g)}]&=p!q!\innp{\tilde{f},\tilde{g}},\\
   E[\abs{I_{p,q}(f) }^2]&=p!q!\norm{\tilde{f}}^2\le p!q!\norm{f}^2,\text{\quad (It\^{o}'s isometry)}\label{symmetric2}
\end{align} where $\norm{\cdot}$ and $\innp{\cdot,\cdot}$ are norm and inner product on $\FH^{\otimes(p+ q)} $.
Since $\mathcal{S}_{pq}$ is dense in $\FH^{\otimes(p+ q)} $, one can extend the integral to any $f\in\FH^{\otimes(p+ q)} $ by taking the limit, i.e.,
\begin{equation}
   I_{p,q}(f):=\int\cdots\int f \dif M(t_{1})\dots \dif M(t_{p})\overline{\dif M(s_{1})}\dots \overline{ \dif M(s_{q})}=\lim_{n} I_{p,q}(f_n),
\end{equation}
where $f_n\in \mathcal{S}_{pq}$ such that $f_n\to f$ in $\FH^{\otimes(p+ q)} $, and the definition is independent of the choice of the sequence $\set{f_n}$. In addition, (\ref{ipq}-\ref{symmetric2}) are still valid to any $f,g\in \FH^{\otimes(p+ q)}$.
Moreover, the set
\begin{equation*}
  \mathscr{H}_{p,q}:=\set{I_{p,q}(f):\,f\in \FH^{\otimes(p+ q)}}
\end{equation*}
is called the Wiener-It\^{o} chaos of degree of $(p,q)$ or $(p,q)$-th Wiener-It\^{o} chaos.
 \end{defn}

\begin{defn}{\bf (Complex Hermite polynomials)}
The complex Hermite polynomials $ J_{m,n}(z,\rho)$ are given by \cite{ito}
\begin{equation}\label{itldefn}
    \exp\set{\lambda \bar{z} + \bar{\lambda}z-\rho |\lambda|^2}=\sum_{m=0}^{\infty}\sum_{n=0}^{\infty}\frac{\bar{\lambda}^m\lambda^n}{m!n!}J_{m,n}(z,\rho),
\end{equation}  where $\lambda\in \Cnum$. When $\rho=2$, we will often write $J_{m,n}(z)$ rather than $J_{m,n}(z,\rho)$.
\end{defn}
Applying \cite[Theorem~9]{ito} (or see Lemma~\ref{lmito} below) and the properties of complex Hermite polynomials (see \cite[Theorem 12]{ito}), It\^{o} established the relation between complex multiple integrals and complex Hermite polynomials \cite[Theorem13.2]{ito}: suppose that $ {h}_1(t),\dots,{h}_l(t)$ be any orthonormal system in $\FH$ and $\alpha_i,\,\beta_j=1,\dots,l$, then
\begin{align}
  &\int\cdots\int {h}_{\alpha_1}(t_1)\cdots{h}_{\alpha_m}(t_m)\overline{{h}_{\beta_1}(s_1)}\dots\overline{{h}_{\beta_n}(s_n)} \dif M(t_{1})\dots \dif M(t_{m})\overline{\dif M(s_{1})}\dots \overline{ \dif M(s_{n})}\nonumber \\
  &=\prod_{k=1}^l \, {2^{-\frac{m_k+n_k}{2}}}J_{m_k,n_k}(\sqrt2 Z(h_k)),\label{itomulitp}
\end{align}
where
\begin{equation}\label{zh}
  Z(h_k)=\int {h}_k(t) \dif {M}(t),\,k=1,\dots,l,
\end{equation}
and $m_k,\,n_k$ are the number of $k$ appeared in $\alpha_i$ and $\beta_j$ respectively.
\begin{rem}\label{rm2}
 As a result of the above equality and Proposition~2.9 of \cite{cl2}, $\mathscr{H}_{p,q} $ is equal to the the closed linear subspace of $L_{\Cnum}^2(\mathbf{M})$ generated by the random variables of the type
\begin{equation}\label{chaos}
\set{J_{m,n}(Z({h})), {h}\in \mathfrak{H},\norm{{h}}_{\mathfrak{H}}=\sqrt2},
\end{equation}where $Z(h)$ is the same as (\ref{zh}).
Please refer to Definition~2.7 and Remark~9 of \cite{cl2} for details.
\end{rem}
\begin{nott}\label{com conjut}
   Suppose $  f(t_1,\dots,t_p,s_1,\dots,s_q)\in \FH^{\odot p}\otimes \FH^{\odot q}$. We call $$ \FH^{\odot q}\otimes \FH^{\odot p}\ni h(t_1,\dots,t_q,s_1,\dots,s_p):=\bar{f}(s_1,\dots,s_p,t_1,\dots,t_q)$$ {\bf the reversed complex conjugate} of function $  f(t_1,\dots,t_p,s_1,\dots,s_q)$.
\end{nott}
From Definition~\ref{dff}, we can obtain the following lemma easily.
\begin{lem}\label{lm1}
   Suppose $  f(t_1,\dots,t_p,s_1,\dots,s_q)\in \FH^{\odot p}\otimes \FH^{\odot q}$. Let $h$ be the reversed complex conjugate of $f$, then
\begin{equation}\label{oli}
   \overline{{I}_{p,q}(f)}={I}_{q,p}(h).
\end{equation}
\end{lem}

We conclude these preliminaries by two propositions, that will be needed throughout the sequel.
\begin{prop}\label{pp201}
    \begin{itemize}
        \item[\textup{(1)}] Complex multiple Wiener-It\^{o} integrals have all moments satisfying the following hypercontractivity inequality
        \begin{equation}\label{hypercontract}
         [E\abs{I_{p,q}(f)}^{r}]^{\frac{1}{r}}\le (r-1)^{\frac{p+q}{2}} [E\abs{I_{p,q}(f)}^{2}]^{\frac{1}{2}},\quad r\ge 2,
        \end{equation}
         where $\abs{\cdot}$ is the absolute value (or modulus) of a complex number.
        \item[\textup{(2)}] If a sequence of distributions of $\set{I_{p,q}(f_n)}_{n\ge 1}$ is tight, then
        \begin{equation}\label{tight compact}
          \sup_{n} E\abs{I_{p,q}(f)}^{r}<\infty \quad\text {for every $r>0$.}
        \end{equation}
    \end{itemize}
\end{prop}
\begin{proof}
(i) (\ref{hypercontract}) is a consequence of the hypercontractivity of normal Ornstein-Uhlenbeck semigroup \cite{cy}.  \\
(ii) Along the same line as (ii) of \cite[Lemma 2.1]{nr} for the case of real multiple integrals, we can show that (\ref{tight compact}) holds.
\end{proof}

Set $\mathbf{M}=\frac{1}{\sqrt2}[\mathbf{M}_1+\mi \mathbf{M}_2]$, $\mathbf{M}_1,\, \mathbf{M}_2$ are two real independent continuous normal system. Let $\widehat{T}=\set{1,2}\times T,\,\mathcal{B}(\widehat{T})=\mathcal{B}(\set{1,2}\times {T})$,
\begin{equation*}
  \widehat{{M}}(B)={M}_1(B_1)+ {M}_2(B_2),\quad \forall B=\big(\set{1}\times B_1 \big) \bigcup \big(\set{2}\times B_2 \big)\in \mathcal{B}(\widehat{T}).
\end{equation*}
Then $ L^2(\widehat{T}) =L^2(T)\oplus L^2(T)$ and $$\widehat{\mathbf{M}}=\set{\widehat{{M}}(B):\,B=\big(\set{1}\times B_1 \big) \bigcup \big(\set{2}\times B_2 \big),\, \mu(B_1)+\mu(B_2)<\infty} $$ is a real normal random measure on $(\hat{T},\,\mathcal{B}(\widehat{T}))$. Denote by $\mathrm{I}_{n}(f)$ the real $n$-th multiple Wiener-
It\^{o} integral of $f$ with respect to $ \widehat{\mathbf{M}}$ (see subsection~3.2 of \cite{cl2}).
\begin{prop}\label{prop25}
 Suppose that $h\in \FH^{\otimes p}\otimes \FH^{\otimes q}$. Then there exist $f,g\in (L^2(\widehat{T}))^{\otimes (p+q)}$ such that
\begin{equation}\label{cll}
  I_{p,q}(h)=\mathrm{I}_{p+q}(f)+\mi\, \mathrm{I}_{p+q}(g).
\end{equation}
That is, both of the real part and the imaginary part of a complex multiple integral can be represented by real multiple
integrals \cite[Theorem 3.3]{cl2}.
\end{prop}
\section{The product formula for complex multiple Wiener-It\^{o} integrals}

\begin{nott}\label{contra}
For two symmetric functions $f\in \FH^{\odot p_1}\otimes \FH^{\odot q_1},\,g\in \FH^{\odot p_2}\otimes \FH^{\odot q_2}$ and $i\le p_1\wedge q_2,\,j\le q_1\wedge p_2$, the contraction of $(i,j)$ indices of the two functions is given by
   \begin{align}
    & f\otimes_{i,j} g (t_1,\dots,t_{p_1+p_2-i-j
    };s_1,\dots,s_{q_1+q_2-i-j})\label{comprt}\\
    &=\int_{A^{i+j}}\mu^{i+j}(\dif u_1\cdots \dif u_i\dif v_1\cdots \dif v_j)\,f(t_1,\dots,t_{p_1-i},u_1,\dots,u_i; s_1\dots,s_{q_1-j},v_1\dots,v_j)\nonumber\\
    & \times g(t_{p_1-i+1},\dots,t_{p_1-i+p_2-j},v_1\dots,v_j; s_{q_1-j+1}\dots,s_{q_1-j+q_2-i},u_1\dots,u_i);\nonumber
   \end{align}
   by convention, $f\otimes_{0,0} g=f\otimes g$ denotes the tensor product of $f$ and $g$.
We write $f\tilde{\otimes}_{p,q} g$ for the symmetrization of $f\otimes_{p,q} g$. In what follows, we use the convention $f\otimes_{i,j} g=0$ if $i> p_1\wedge q_2$ or $j> q_1\wedge p_2$.
\end{nott}
The following lemma is the starting point of the relationship (\ref{itomulitp}) and the product formula for complex multiple Wiener-It\^{o} integrals.
\begin{lem}\label{lmito}\cite[Theorem~9]{ito}
   Let $f\in \FH^{\odot p}\otimes \FH^{\odot q} $ be a symmetric function and let $g\in \FH$. Then
   \begin{align}
      I_{p,q}(f)I_{1,0}(g)&=I_{p+1,q}(f\otimes g)+qI_{p,q-1}(f\otimes_{0,1} g),\label{ito001}\\
      I_{p,q}(f)I_{0,1}(g)&=I_{p,q+1}(f\otimes g)+pI_{p-1,q}(f\otimes_{1,0} g).\label{ito002}
   \end{align}
\end{lem}

\begin{thm}{\bf (Product formula)}\label{pdt fml}
For two symmetric functions $f\in \FH^{\odot a}\otimes \FH^{\odot b},\,g\in \FH^{\odot c}\otimes \FH^{\odot d}$,
the product formula for complex multiple Wiener-It\^{o} integrals is given by
   \begin{align}\label{ii}
     I_{a,b}(f)I_{c,d}(g)&=\sum_{i=0}^{a\wedge d}\sum^{b\wedge c}_{j=0}\,{a\choose i}{d\choose i}{b\choose j}{c\choose j}i!j!\,I_{a+c-i-j,b+d-i-j}(f\otimes_{i,j}g).
   \end{align}
   where $a,b,c,d\in \Nnum$.
\end{thm}
\begin{proof}
From Remark~\ref{rm2},
we only need to show (\ref{ii}) hold for $I_{a,b}(f)=J_{a,b}(Z(h_1))$ and $I_{c,d}(g)=J_{c,d}(Z(h_2))$ with $h_1,h_2\in \FH$ such that $\norm{h_1}=\norm{h_2}=\sqrt2$ by a density argument. That is to say, $f=h_1^{\otimes a}\otimes \bar{h}_1^{\otimes b},\,g=h_2^{\otimes c}\otimes \bar{h}_2^{\otimes d}$.
By the decomposition theorem of Hilbert spaces\cite[p71]{riesz}, we may as well assume that $h_1=h_2$ or $\innp{h_1,\,h_2}_{\FH}=0$.

It follows from the generating function of complex Hermite polynomials \cite{cl,ito} that
\begin{align*}
& \sum_{a=0}^{\infty} \sum_{b=0}^{\infty}\frac{\bar{\lambda}^a\lambda^b}{a!b!}J_{a,b}(z)\sum_{c=0}^{\infty}\sum_{d=0}^{\infty}\frac{\bar{\mu}^c\mu^d}{m!n!}J_{c,d}(z)  \\
 &= \exp\set{\lambda \bar{z} + \bar{\lambda}z-2 |\lambda|^2}  \exp\set{\mu \bar{z} + \bar{\mu}z-2 |\lambda|^2}\\
 &=\exp\set{(\lambda+\mu)\bar{z}+\overline{(\lambda+\mu)}z-2 |\lambda+\mu|^2}\exp\set{2(\bar{\lambda}\mu+\lambda\bar{\mu})}\\
 &=\sum_{m=0}^{\infty}\sum_{n=0}^{\infty}\frac{\overline{(\lambda+\mu)^m}(\lambda+\mu)^n}{m!n!}J_{m,n}(z)\sum_{i=0}^{\infty}\sum_{j=0}^{\infty}
 \frac{2^{i+j}(\bar{\lambda}\mu)^i(\lambda\bar{\mu})^j}{i!j!},
\end{align*} where $z,\lambda,\mu \in \Cnum$.
Expanding $(\lambda+\mu)^n$ and comparing coefficients immediately yield:
   \begin{align}\label{hh}
      J_{a,b}(z)J_{c,d}(z)=\sum_{i=0}^{a\wedge d}\sum^{b\wedge c}_{j=0}\,{a\choose i}{d\choose i}{b\choose j}{c\choose j}i!j!2^{i+j}\,J_{a+c-i-j,b+d-i-j}(z).
   \end{align} 
When $h_1=h_2$, it follows from the relationship (\ref{itomulitp}) that (\ref{ii}) is exact (\ref{hh}).

When $\innp{h_1,\,h_2}_{\FH}=0$, we have that
\begin{equation*}
 f\otimes_{i,j}g=  \left\{
      \begin{array}{ll}
     0, & i+ j>0,\\
     (h_1^{\otimes a} \otimes h_2^{\otimes c})\otimes (\bar{h}_1^{\otimes b}\otimes \bar{h}_2^{\otimes d}), & i=j=0. \\
      \end{array}\nonumber \right.
\end{equation*}
Thus, (\ref{ii}) is degenerated to the relationship (\ref{itomulitp}) in this case.
\end{proof}
\begin{rem}
There are two another approaches to show Theorem~\ref{pdt fml}. One is by induction over the indices $c$ and $d$ (see \cite[Proposition 1.1.3]{Nua}) using Lemma~\ref{lmito}, which involves some tedious combinatorial calculations. The other is by Malliavin calculus (see \cite[Theorem 2.7.10]{np}) if we exploit the framework of complex Malliavin operators.
\end{rem}

The following product formula which is a direct corollary of Lemma~\ref{lm1} and Theorem~\ref{pdt fml}, will be used later.
\begin{cor}\label{ii2}
      \begin{align*}
     I_{a,b}(f)\overline{I_{c,d}(g)}&=\sum_{i=0}^{a\wedge c}\sum^{b\wedge d}_{j=0}\,{a\choose i}{c\choose i}{b\choose j}{d\choose j}i!j!\,I_{a+d-i-j,b+c-i-j}(f\otimes_{i,j}h),
   \end{align*}
   where $h$ is the reversed complex conjugate of $g$ (see Notation~\ref{com conjut}), and
   \begin{align}
    & f\otimes_{i,j} h (t_1,\dots,t_{a+d-i-j
    };s_1,\dots,s_{b+c-i-j})\nonumber \\
    &=\int_{A^{i+j}}\mu^{i+j}(\dif u_1\cdots \dif u_i\dif v_1\cdots \dif v_j)\,f(t_1,\dots,t_{a-i},u_1,\dots,u_i; s_1,\dots,s_{b-j},v_1\dots,v_j)\nonumber\\
    & \times \bar{g}( s_{b-j+1},\dots,s_{b-j+c-i},u_1,\dots,u_i ;t_{a-i+1},\dots,t_{a-i+d-j},v_1,\dots,v_j).\label{fh}
   \end{align}
\end{cor}

\section{The independence of complex multiple Wiener-It\^{o} integrals}
\begin{lem}\label{lm31}
   For two symmetric functions $f\in \FH^{\odot a}\otimes \FH^{\odot b},\,g\in \FH^{\odot c}\otimes \FH^{\odot d}$, let $F=I_{a,b}(f),\,G=I_{c,d}(g)$ and $h$ be the reversed complex conjugate of $g$. Then
\begin{align*}
     & \mathrm{Cov}(|F|^2,\,|G|^2)\nonumber\\
      &=\sum_{i+j>0}{a\choose i}{c\choose i}{b\choose j}{d\choose j} a!b!c!d! \norm{f\otimes_{i,j}g}^2_{\FH^{\otimes(m-2(i+j))}}\nonumber\\
      &+\sum_{r\ge 1}\big((a+d-r)!(b+c-r)!\big)^2 \norm{\phi_r}^2_{\FH^{\otimes(m-2r)}},
\end{align*}
   where 
\begin{align}\label{bbb}
   \phi_r=\sum_{i+j=r}{a\choose i}{c\choose i} {b\choose j} {d\choose j}i!j!  f\tilde{\otimes}_{i,j} h.
\end{align}
As a consequence, the squares of the absolute values of complex multiple Wiener-It\^{o} integrals are non-negatively correlated.
\end{lem}
\begin{proof}
   We divide the proof into three steps. Let $m=a+b+c+d$.

   Firstly, it follows from Corollary~\ref{ii2}, the orthogonal property and It\^{o}'s isometry of multiple Wiener-It\^{o} integrals that
   \begin{align*}
      E[|F\bar{G}|^2]=\sum_{r\ge 0}\big((a+d-r)!(b+c-r)!\big)^2\norm{\phi_r}^2_{\FH^{\otimes(m-2r)}}.
   \end{align*}

   Secondly, we claim that
   \begin{align*}
      (a+d)!(b+c)!\norm{f\tilde{\otimes}h}^2=\sum_{i=0}^{a\wedge d}\sum_{j=0}^{b\wedge c} {a\choose i}{d\choose i}{b\choose j}{c\choose j} a!b!c!d! \norm{f\otimes_{i,j}g}^2_{\FH^{\otimes(m-2(i+j))}}.
   \end{align*}
  Let $\pi$ ($\sigma$ resp.) be a permutation of the set $\set{1,\dots, a+d}$ ($\set{1,\dots,b+c} $ resp.). Denote by $\pi_0,\,\sigma_0$ the identity permutations. We write $\pi\sim_{i} \pi_0$ ($\sigma\sim_{i} \sigma_0$ resp.) if the set $\set{\pi(1),\dots,\pi(a)}\cap \set{1,\dots, a} $
  ($\set{\sigma(1),\dots,\pi(b)}\cap \set{1,\dots, b}$ resp.) contains exactly $i$ elements \cite{NuaPec}.  Then we have that
   \begin{align*}
     &(a+d)!(b+c)!\norm{f\tilde{\otimes}h}^2\\
     &=(a+d)!(b+c)!\innp{f{\otimes}h,\,f\tilde{\otimes}h}\\
     &=\sum_{\pi}\sum_{\sigma}\int_{A^m}\dif \mu^{m}f(t_1,\dots,t_a,s_1\dots,s_b)\bar{g}(s_{b+1},\dots,s_{b+c},t_{a+1},\dots,t_{a+d})\\
     &\times \bar{f}(t_{\pi(1)},\dots,t_{\pi(a)},s_{\sigma(1)}\dots,s_{\sigma(b)}) g(s_{\sigma(b+1)},\dots,s_{\sigma(b+c)},t_{\pi(a+1)},\dots,t_{\pi(a+d)})\\
     &=\sum_{i=0}^{a\wedge d}\sum_{j=0}^{b\wedge c}\sum_{\pi\sim_{a-i} \pi_0}\sum_{\sigma\sim_{b-j} \sigma_0}\norm{f\otimes_{i,j}g}^2_{\FH^{\otimes(m-2(i+j))}}\\
     &=\sum_{i=0}^{a\wedge d}\sum_{j=0}^{b\wedge c} {a\choose i}{d\choose i}{b\choose j}{c\choose j} a!b!c!d! \norm{f\otimes_{i,j}g}^2_{\FH^{\otimes(m-2(i+j))}}.
   \end{align*}

   Thirdly, note that when $i=j=0$, $\norm{f\otimes_{i,j}g}^2_{\FH^{\otimes(m-2(i+j))}}=\norm{f}^2\norm{g}^2$. It\^{o}'s isometry implies that
   $E[|F|^2]E[|G|^2]= a!b!c!d! \norm{f}^2\norm{g}^2$. Thus,
   \begin{align*}
      &\mathrm{Cov}(|F|^2,\,|G|^2)\\
      &=E[|F\bar{G}|^2]-E[|F|^2]E[|G|^2]\\
      &=\sum_{i+j>0}{a\choose i}{c\choose i}{b\choose j}{d\choose j} a!b!c!d! \norm{f\otimes_{i,j}g}^2_{\FH^{\otimes(m-2(i+j))}}\\
      &+\sum_{r\ge 1}\big((a+d-r)!(b+c-r)!\big)^2 \norm{\phi_r}^2_{\FH^{\otimes(m-2r)}}.
   \end{align*}
\end{proof}

\begin{thm}\label{mthm}{\bf (\"{U}st\"{u}nel-Zakai independence criterion) }\label{thm2}
   For two symmetric functions $f\in \FH^{\odot a}\otimes \FH^{\odot b},\,g\in \FH^{\odot c}\otimes \FH^{\odot d}$ with $a+b\ge 1 ,\,c+d\ge 1$, the following conditions are equivalent:
     \begin{itemize}
    \item[\textup{(i)}] $I_{a,b}(f)$ and $I_{c,d}(g)$ are independent random variables;
    \item[\textup{(ii)}] $f\otimes_{1,0}g=0$, $f\otimes_{0,1}g=0$, $f\otimes_{1,0}h=0$ and $f\otimes_{0,1}h=0$ a.e. $\mu^{m-2}$, where $m=a+b+c+d$ and $h$ is the reversed complex conjugate of $g$.
  \end{itemize}
\end{thm}

\begin{proof}
(i)$\Rightarrow$ (ii): Denote $F=I_{a,b}(f),\,G=I_{c,d}(g)$. It follows from Proposition~\ref{pp201} (1) that inside a fixed Wiener chaos (i.e., for the fixed $(a,b)$), all the $L^q$-norms ($q>1$) are equivalent. Thus $\mathrm{Cov}(\abs{F}^2,\,\abs{G}^2)$ is finite. If (i) is satisfied then $\mathrm{Cov}(\abs{F}^2,\,\abs{G}^2)=0$. It follows from Lemma~\ref{lm31} that $f\otimes_{1,0}g=0$ and $f\otimes_{0,1}g=0$. Note that $\bar{G}=I_{d,c}(h)$ and $F,\,\bar{G}$ are also independent random variables. Thus we also have that $f\otimes_{1,0}h=0$ and $f\otimes_{0,1}h=0$.

(ii)$\Rightarrow$ (i): We divide the proof into three steps along the same line as the proof for real multiple integrals \cite{klb}.

Firstly, let $\mathcal{H}_f,\,\mathcal{G}_f $ respectively denote the Hilbert subspace in $\FH$ spanned by all functions
\begin{align*}
   t\mapsto \int_{A^{a+b-1}}f(t,x_1,\dots,x_{a+b-1})h(x_1,\dots,x_{a+b-1})\mu^{a+b-1}(\dif x_1\dots \dif x_{a+b-1})\\
   t\mapsto \int_{A^{a+b-1}}f(x_1,\dots,x_{a+b-1},t)h(x_1,\dots,x_{a+b-1})\mu^{a+b-1}(\dif x_1\dots \dif x_{a+b-1})
\end{align*}
 where $t\in A$ and $h\in \FH^{\otimes (a+b-1)}$. Similarly we define $\mathcal{H}_g,\,\mathcal{G}_g$ in terms of $g$. Denote by $\overline{\mathcal{G}_f} $ the complex conjugate of $\mathcal{G}_f$. We claim that condition (ii) implies that $\set{\mathcal{H}_f,\,\overline{\mathcal{G}_f}}$ and $\set{\mathcal{H}_g,\,\overline{\mathcal{G}_g}}$ are orthogonal. In fact, $f\otimes_{1,0}g=0$, $f\otimes_{0,1}g=0$, $f\otimes_{1,0}h=0$ and $f\otimes_{0,1}h=0$ respectively imply that $\mathcal{H}_f \perp \overline{\mathcal{G}_g}$, $\overline{\mathcal{G}_f}\perp \mathcal{H}_g$, $\mathcal{H}_f\perp \mathcal{H}_g $ and $\overline{\mathcal{G}_f}\perp \overline{\mathcal{G}_g}$ using Fubini Theorem. For example, let $h(x)\in \FH^{\otimes (a+b-1)},\,l(y)\in \FH^{\otimes (c+d-1)}$ and $m=a+b+c+d$, we have that
 \begin{align*}
    &\int_A\mu(\dif t)\int_{A^{a+b-1}}f(t,x)h(x)\mu^{a+b-1}(\dif x)\int_{A^{c+d-1}}g(y,t)l(y)\mu^{c+d-1}(\dif y)\\
    &=\int_{A^{m-2}}\mu^{m-2}(\dif x\dif y)h(x)l(y)\int_A f(t,x)g(y,t)\mu(\dif t)\\
    &=\int_{A^{m-2}}h(x)l(y)f\otimes_{1,0}g\,\mu^{m-2}(\dif x\dif y)=0
 \end{align*}
 i.e., $\mathcal{H}_f $ and $\overline{\mathcal{G}_g}$ are orthogonal.

Secondly, let $\set{\varphi_n}$ ($\set{\psi_n}$ resp.) be an orthonormal basis for the Hilbert subspace in $\FH$ spanned by $\set{\mathcal{H}_f,\,\overline{\mathcal{G}_f}}$ ( $\set{\mathcal{H}_g,\,\overline{\mathcal{G}_g}}$ resp.). Since the tensor products $\varphi_{i_1}\otimes \cdots\otimes \varphi_{i_a}\otimes \bar{\varphi}_{j_1}\otimes\cdots\otimes \bar{\varphi}_{j_b}$ form an orthonormal basis in $\mathcal{H}^{\otimes a}\otimes \overline{\mathcal{G}_f}^{\otimes b}$, it follows from monotonic class theorem \cite{chung,durrett} that $f$ (and $g$)
can be decomposed as
\begin{align*}
   f&=\sum e_{i_1\dots i_a j_1\dots j_b} \varphi_{i_1}\otimes \cdots\otimes \varphi_{i_a}\otimes \bar{\varphi}_{j_1}\otimes\cdots\otimes \bar{\varphi}_{j_b},\\
   g&=\sum l_{i_1\dots i_c j_1\dots j_d} \psi_{i_1}\otimes \cdots\otimes \psi_{i_c}\otimes \bar{\psi}_{j_1}\otimes\cdots\otimes \bar{\psi}_{j_d}.
\end{align*}

Thirdly, we claim that $F,G$ are independent. In fact, we write $\xi_i=\int_A \varphi_i(a)M(\dif a) $ and $\eta_j=\int_A \psi_j M(\dif a)$ for all $i$ and $j$. Then the Cram\'{e}r-Wold theorem implies that the entire sequences $\set{\xi_i}$ and $\set{\eta_j}$ are independent \cite{klb}. It follows from (\ref{itomulitp}) that $F=I_{a,b}(f)$ (  $G=I_{c,d}(f)$ resp.) can be expanded into polynomials in $\xi_1,\xi_2,\dots$ ( $\eta_1,\eta_2,\dots$ resp.). Thus $F,\,G$ are independent.
\end{proof}
\begin{rem} Similar to real multiple integrals \cite{rs}, condition (i) of Theorem~\ref{mthm} is also equivalent to
   \begin{itemize}
      \item[\textup{(iii)}] $\mathrm{Cov}(\abs{F}^2,\,\abs{G}^2)=0$, i.e., $\abs{F}^2,\,\abs{G}^2$ are uncorrected,
   \end{itemize}
which can be observed from the above proof.
\end{rem}

\section{Appendix: Asymptotic independence of complex multiple Wiener-It\^{o} integrals}
\begin{defn}\cite[Definition 3.3]{nr}
Fix $d\ge 1$ and for each $n\ge 1$, let $F_n=(F_{1,n},\dots,F_{d,n})$ be a $d$-dimensional complex-valued random variable. We say
the variables $(F_{i,n})_{1\le i\le d}$ are asymptotically moment-independent if each $F_{i,n}$ admits moments of all orders and for any two sequences $(l_1,\dots,l_d)$ and $(k_1,\dots,k_d)$ of non-negative integrals,
\begin{equation}\label{asympot indep}
 \lim_{n\to \infty}\set{E[\prod_{i=1}^d F_{i,n}^{l_i}\bar{F}_{i,n}^{k_i}]-\prod_{i=1}^d E[F_{i,n}^{l_i}\bar{F}_{i,n}^{k_i}]}=0.
\end{equation}
\end{defn}
\begin{thm} {\bf (Nourdin-Rosi\'{n}ski asymptotically moment-independence criterion)}\label{thm3}
Fix $d\ge 2$ and let $(a_1,\dots,a_d)$ and $(b_1,\dots,b_d)$ be two sequences of non-negative integrals.
   For each $n\ge 1$, let $F_n=(F_{1,n},\dots,F_{d,n})$ be a $d$-dimensional complex multiple Wiener-It\^{o} integrals, where
   $F_{i,n}=I_{a_i,b_i}(f_{i,n})$ with $f_{i,n}\in \FH^{\odot a_i}\otimes\FH^{\odot b_i}$. If for every $1\le i\le d$,
   \begin{equation}\label{tight}
     \sup_{n}E[\abs{F_{i,n}}^2]<\infty,
   \end{equation}
   then the following conditions are equivalent:
     \begin{itemize}
    \item[\textup{(i)}] the random variables $(F_{i,n})_{1\le i \le d}$ are asymptotically moment-independent;
    \item[\textup{(ii)}] $\lim_{n\to\infty}\mathrm{Cov}(\abs{F_{i,n}}^2,\,\abs{F_{j,n}}^2)=0$ for every $i\neq j$;
    \item[\textup{(iii)}]For  every $i\neq j$, $\lim_{n\to\infty}\norm{f_{i,n}\otimes_{r,s}f_{j,n}}=0$ for every (r,s) such that $r\le a_i\wedge b_j,\,s\le a_j\wedge b_i,\,r+s>0$, and $\lim_{n\to\infty}\norm{f_{i,n}\otimes_{r,s}h_{j,n}}=0$ for every (r,s) such that $r\le a_i\wedge a_j,\,s\le b_i\wedge b_j,\,r+s>0$, where $h_{j,n}$ is the reversed complex conjugate of $f_{j,n}$.
  \end{itemize}
\end{thm}
\begin{proof} Suppose that $F_{i,n}=U_{i,n}+\sqrt{-1}V_{i,n}$.
 Thus, it follows from Theorem~3.4 and Remark 3.5 of \cite{nr} and Proposition~\ref{prop25} that the random vectors $(U_{i,n},V_{i,n}),\,i=1,\dots,d $ being asymptotically moment-independent, i.e., for any sequence $(l_1,\dots,l_d)$ and $(k_1,\dots,k_d)$,
\begin{equation}\label{lk}
    \lim_{n\to \infty}\set{E[\prod_{i=1}^d U_{i,n}^{l_i}V_{i,n}^{k_i}]-\prod_{i=1}^d E[U_{i,n}^{l_i}V_{i,n}^{k_i}]}=0
\end{equation}
is equivalent to that $\lim_{n\to \infty}\mathrm{Cov}(\abs{F_{i,n}}^2,\,\abs{F_{j,n}}^2 )=0 $ for every $i\neq j$.
It is easy to check that (\ref{asympot indep}) is equivalent to (\ref{lk}). Thus, (i) is equivalent to (ii).

Using (\ref{daxiao}), it follows from Lemma~\ref{lm31} that (ii) is equivalent to (iii).
\end{proof}
\begin{cor}{\bf (Nourdin-Rosi\'{n}ski joint convergence criterion)}\label{cor4}
  Under notation of Theorem~\ref{thm3}, let $(\eta_i)_{1\le i\le d}$ be a complex random vector such that
    \begin{itemize}
    \item[\textup{(i)}] As $n\to \infty$, $F_{i,n}$ converges in law to $\eta_i$ for each $1\le i\le d$;
    \item[\textup{(ii)}]The random variables $\eta_1,\dots,\eta_d$ are independent;
    \item[\textup{(iii)}] Condition (ii) or (iii) of Theorem~\ref{thm3} holds;
    \item[\textup{(iv)}] The law of $\eta_i$ is determined by its moments for each $1\le i\le d$.
  \end{itemize}
  Then the joint convergence holds, i.e.,
  \begin{equation}\label{limmmtt}
     (F_{1,n},\dots,F_{d,n}) \stackrel{\rm law}{ \longrightarrow} (\eta_1,\dots,\eta_d),\quad \text{  as  } \quad n\to \infty.
  \end{equation}
\end{cor}
\begin{rem}
 In the above condition (iv), we do not assume that the laws of both the real part and the imaginary part of $\eta_i$ are determined by their moments. This is the difference between Corollary~3.6 of \cite{nr} and Corollary~\ref{cor4}.
\end{rem}
\begin{proof}
Note that Theorem~3 of \cite{peter} still holds for probability measures on $\Cnum^d $. Precisely stated, if each of $d$ coordinate projections $P_i(\mu)$ of a probability measure $\mu$ on $\Cnum^d $ is uniquely determined by its moments sequence, then the measure $\mu$ is also uniquely determined by its moments. By applying Proposition~\ref{pp201} (2), we can show the desired conclusion by means of modifying the proof of Corollary~3.6 of \cite{nr} slightly.
\end{proof}
In analogy with the characterization of moment-independence of limits of real multiple Wiener-It\^{o} integrals \cite[Theorem 3.1]{nr}, we can deduce the following corollary for complex multiple Wiener-It\^{o} integrals from the above proof.
\begin{cor}{\bf (Nourdin-Rosi\'{n}ski moment-independence criterion of limits)}\label{cor5}
    Under notations of Theorem~\ref{thm3} and Corollary~\ref{cor4}. Assume that (\ref{limmmtt}) holds. Then $\eta_i$'s admit moments of all orders and the following conditions are equivalent:
     \begin{itemize}
    \item[\textup{($\alpha$)}] The random variables $(\eta_i)_{1\le i \le d}$ are moment-independent, i.e.,
    $$E[\prod_{i=1}^d \eta_{i}^{l_i}\bar{\eta}_{i}^{k_i}]= \prod_{i=1}^d E[\eta_{i}^{l_i}\bar{\eta}_{i}^{k_i}],\quad  \forall  k_1,\dots,k_d,\,l_1,\dots,l_d \in \Nnum;$$
    \item[\textup{($\beta$)}] $\lim_{n\to\infty}\mathrm{Cov}(\abs{F_{i,n}}^2,\,\abs{F_{j,n}}^2)=0$ for every $i\neq j$.
  \end{itemize}
Moreover, if condition (iv) of Corollary~\ref{cor4} is satisfied, then ($\alpha$) is equivalent to that
    \begin{itemize}
    \item[\textup{($\gamma$)}] The random variables $(\eta_i)_{1\le i \le d}$ are independent.
  \end{itemize}
\end{cor}

\vskip 0.2cm {\small {\bf  Acknowledgements}}\   This work was
supported by  NSFC( No.11101137) and Natural Science Foundation of Hunan Province (No.2015JJ2055).

\end{document}